\documentclass[12pt]{article}

\usepackage{amsmath,amsthm,amsfonts}
\usepackage[margin=1in]{geometry}
\usepackage[T1]{fontenc}
\usepackage{booktabs}
\usepackage{color}
\definecolor{webgreen}{rgb}{0,.5,0}
\definecolor{webbrown}{rgb}{.6,0,0}
\usepackage[colorlinks=true,
linkcolor=webgreen,
filecolor=webbrown,
citecolor=webgreen]{hyperref}

\theoremstyle{plain}
\newtheorem{theorem}{Theorem}
\newtheorem{corollary}[theorem]{Corollary}
\newtheorem{lemma}[theorem]{Lemma}

\theoremstyle{definition}
\newtheorem{definition}[theorem]{Definition}
\newtheorem{example}[theorem]{Example}
\newtheorem{conjecture}[theorem]{Conjecture}

\theoremstyle{remark}

\newcommand{\bbN}{\mathbb{N}}
\newcommand{\bbZ}{\mathbb{Z}}
\newcommand{\adics}{\bbZ_{2}}
\newcommand{\Fq}{F_2[[q]]}
\newcommand{\Set}[1]{\left\{#1\right\}}
\newcommand{\Tab}{T_{A,B}}
\newcommand{\Tqq}{T_{1,1+q^2}}
\newcommand{\xcong}[1]{\underset{#1}{\equiv}}
\newcommand{\xncong}[1]{\underset{#1}{\not\equiv}}
\newcommand{\xiff}{\Leftrightarrow}
\newcommand{\ximplies}{\Rightarrow}
\newcommand{\xlimplies}{\Leftarrow}
\newcommand{\wPV}{\widehat{\PV}}
\DeclareMathOperator{\PV}{PV}

\begin{document}

\vspace*{2em}
\begin{center}
{\LARGE\bf 
On $q$-Analogs of the $3x+1$ Dynamical System}
\vskip 1cm
\large
Kenneth G.\ Monks \\
Department of Mathematics \\
University of Scranton \\
Scranton, PA 18510 \\
USA\\
\href{mailto:monks@scranton.edu}{\tt monks@scranton.edu} \\
\end{center}

\vskip .2 in
\begin{abstract}
We construct an infinite family of maps on $F_2[[q]]$ that are conjugate to the famous $3x+1$ map on the $2$-adic integers.  We identify one such map that has the further property that the positive integers whose $3x+1$-orbit contains $1$ correspond to polynomials via a conjugacy, and whose polynomial orbits all eventually enter the $2$-cycle or one of the two fixed points. Proving that every positive integer corresponds to a polynomial via this conjugacy would settle the original $3x+1$ conjecture.
\end{abstract}

\section{Introduction and background}

A \emph{discrete dynamical system} on a set $X$ is obtained by iterating a map $f\colon X\rightarrow X$. For any function $f\colon X\rightarrow X$ and any $x\in X$, the \emph{$f$-orbit} of $x$ is the set $\left\{f^{k}(x):k\in\mathbb{N}\right\}$ where $f^{0}$ is the identity map and $f^{k}=f\circ f^{k-1}$ for all positive integers $k$. Two dynamical systems $f\colon X\to X$ and $g\colon Y\to Y$ are \emph{conjugate} (or \emph{isomorphic} as dynamical systems) if and only if there exists a bijection $h\colon X\to Y$ such that $h\circ f=g\circ h$. Such a map $h$ is called a \emph{conjugacy}, and it is easy to see that conjugacies preserve the sizes of orbits as well as other properties of the dynamical system.

The famous $3x+1$ conjecture states that the $T$-orbit of every positive integer contains $1$, where 
    $$
      T(x)  =
        \begin{cases}
           \frac{x}{2}, & \text{if $2$ divides $x$;} \\[0.3em]
           \frac{3x+1}{2}, & \text{otherwise.}
        \end{cases}
    $$
It has been an open problem in mathematics since the 1930s (see Lagarias \cite{Lagarias} and  Wirsching \cite{Wirsching}).

More recently, several authors have studied a similar dynamical system, $T_q\colon F_2[q]\to F_2[q]$ where $F_2[q]$ is the ring of univariate polynomials in $q$ with coefficients in the field with two elements, and for all $x\in F_2[q]$,
  $$
    T_q(x)  =
      \begin{cases}
         \frac{x}{q}, & \text{if $q$ divides $x$;}\\[0.5em]
         \frac{(1+q)x+1}{q}, & \text{otherwise.}
      \end{cases}
  $$
There is a natural bijection $\xi\colon F_2[q]\to\bbZ^+$ that associates a $q$-polynomial $f(q)$ with the positive integer $f(2)$.  Its inverse map sends a positive integer $n$ to the $q$-polynomial obtained by replacing $2$ with $q$ in the base two representation of $n$. Note that the maps $T$ and $T_q$ have the same form if we replace $q$ with $2$ in the definition of $T_q$.

It was shown concurrently by Hicks, Mullen, Yucas, and Zavislak \cite{Hicks} and by Snapp and Tracy \cite{Snapp} that the $T_q$-orbit of any polynomial contains $1$. The latter concluded that the analog of the $3x+1$ conjecture holds for $T_q$ \cite[Theorem 2.1]{Snapp}.  While the result is analogous if we identify $q$-polynomials $p$ with their positive integer $\xi(p)$, this does not prove the $3x+1$ conjecture itself, since iterating $T_q$ on any polynomial eventually reaches the fixed point $1$, but iterating $T$ on infinitely many positive integers eventually reaches the two-cycle $\{1,2\}$. Thus, these two dynamical systems are not conjugate via any correspondence between the positive integers and polynomials.

In this paper, we show that the dynamical systems $T$ and $T_q$ actually are conjugate if we extend their domains to the ring of $2$-adic integers, $\adics$, and the ring of formal power series in $q$, $\Fq$, respectively.  Both dynamical systems have a unique two-cycle. Thus, to prove the $3x+1$ conjecture, it suffices to show for some conjugacy $h$ between $T$ and $T_q$ that for any positive integer $n$, iterating $T_q$ on $h(n)$ eventually reaches the unique two-cycle for $T_q$. Using $h$ instead of $\xi$ produces an analogous conjecture whose resolution would settle the original problem.

A similar situation occurs for any function that is conjugate to $T$. In particular, we will also show that if $A, B$ are elements of $\Fq$ with constant term $1$ (i.e., not divisible by $q$) then the map $\Tab\colon\Fq\to\Fq$ given by 
  $$\Tab(x)=\begin{cases}
                 \frac{x}{q}, & \text{if $q$ divides $x$;}\\[0.3em]
                 \frac{Ax+B}{q}, & \text{ otherwise}
               \end{cases}
  $$
is conjugate to $T$ and the $3x+1$ conjecture is true if and only if for any conjugacy $h\colon \adics\to\Fq$ between $T$ and $\Tab$, and any positive integer $n$, the $\Tab$-orbit of $h(n)$ eventually reaches the unique two-cycle for $\Tab$. Note that $T_q$ is $T_{1+q,1}$ in this notation.

Each element of $\adics$ is a formal sum $\sum_{i=0}^{\infty} x_i 2^i$ where $x_i\in\{0,1\}$ for all $i\in\bbN$. Thus, we can extend $\xi$ to a bijective correspondence between $\Fq$ and $\adics$ by associating a formal power series $f(q)\in\Fq$ to the $2$-adic integer $f(2)$. We will refer to the elements of $\Fq$ as \emph{$q$-series} in what follows for brevity.

The \emph{binary representation} of either a $2$-adic integer $x=\sum_{i=0}^{\infty} x_i 2^i$ or a $q$-series $x=\sum_{i=0}^{\infty} x_i q^i$ is the infinite sequence of zeroes and ones $x_0 x_1 x_2 \cdots$. By identifying $2$-adic integers and $q$-series with their binary representations, we can consider both $T$ and $T_q$ to be dynamical systems from $\Sigma_2\to\Sigma_2$, where $\Sigma_2$ is the set of all infinite binary sequences. We will say that a $2$-adic integer or $q$-series is \emph{even} if and only if the leading digit of its binary representation is $0$ (which implies it is divisible by $2$ or $q$, respectively), and \emph{odd} otherwise. Similarly, the \emph{parity} of an element is either $0$ or $1$ depending on whether the element is even or odd, respectively. 

The ordinary rational numbers with odd denominators (when written as reduced fractions) are isomorphic to a subring of $\adics$ and correspond to the $2$-adic integers with eventually repeating binary representations via the geometric series formula. We can use this as a compact way to represent a $q$-series whose binary representation is eventually repeating (keeping in mind that ordinary addition and multiplication of rational numbers correspond to addition and multiplication in $\adics$, not in $\Fq$). Also note that we can express such power series as rational expressions in $q$ with coefficients in $F_2$.

\begin{example}
   Let $P_1=1+q+q^3+q^4+q^6+q^7+\cdots=\frac{1+q}{1+q^3}\in\Fq$.  The binary expansion of $P_1$ is $110110110\cdots$ which we can abbreviate as $\overline{110}$. But $\overline{110}$ is also the binary representation of the $2$-adic $1+2+2^3+2^4+2^6+2^7+\cdots$ which equals $-3/7$ by the geometric series formula. Thus, we have the correspondence $\xi(P_1)=-3/7$.  Notice that $\xi(f)$ is a positive integer whenever $f\in\Fq$ is a polynomial, which is the correspondence assumed by Snapp and Tracy \cite[Theorem 2.1]{Snapp} when referring to the analogous conjecture.  
\end{example}

Notice that in this situation $f$ is odd if and only if the numerators of both $f$ and $\xi(f)$ are odd when each is written as a reduced fraction.  So the parity of $f$ is always the same as the parity of the numerators of the reduced fractions for $f$ and $\xi(f)$.

\section{Main results}\label{sec:main}

While the correspondence $\xi$ is a natural one to consider when comparing the dynamical systems $T$ and $T_q$, it is not a conjugacy even when their domains are extended to $\adics$ and $\Fq$, respectively.  That such conjugacies do exist, however, is the point of our first main result. For the remainder of this paper, we assume that $T\colon\adics\to\adics$ and $\Tab\colon\Fq\to\Fq$.

\begin{theorem}\label{thm:main}
   For any odd $A, B\in\Fq$, the map $\Tab$ is conjugate to $T$. 
\end{theorem}

\begin{corollary}\label{cor:main1}
   The map $T_q$ is conjugate to $T$.
\end{corollary}

Thus, for any conjugacy $h$ between $T$ and $\Tab$, the power series that correspond to the positive natural numbers are those in $h(\bbN^+)$.  For $T_q$, this set does not include any polynomials, as all nonzero polynomial orbits end in the fixed point $1$ of $T_q$, but no positive integer has a $T$-orbit that contains the fixed points $0$ or $-1$ of $T$.

Similarly, any conjugacy between $T$ and $T_q$ must map $1$ to an element of the unique $2$-cycle of $T_q$, namely $\{P_1,qP_1\}$ (or equivalently $\{-3/7, -6/7\}$). If the conjugacy preserves parity, then for $T_q$ it is $-3/7$ that corresponds to $1$ for $T$, and so the analogous conjecture can be stated as follows:

\begin{conjecture}\label{cor:main2}
  For any conjugacy $h\colon\adics\to\Fq$ between $T$ and $T_q$ and any positive integer $n$, the $T_q$-orbit of $h(n)$ contains $-3/7$.
\end{conjecture}

It follows from the results of Monks and Yazinski \cite{MonksYaz} that there exists a parity-preserving conjugacy $h$ such that $h(1)=-3/7$.  For any such $h$ and any positive integer $n$, we can calculate the corresponding power series $h(n)$ for $T_q$.  If we can show that all such power series have $-3/7$ in their $T_q$ orbit, then it will prove the $3x+1$ conjecture.  We list the first few such values in Table \ref{tbl:conjugates1}.

Given that polynomials do not correspond to positive integers for any conjugacy between $T$ and $T_q$, a natural question to ask is whether such conjugacies might exist for some other maps $\Tab$. Our second main result answers this in the affirmative.

\begin{theorem}\label{thm:main3}
   The $T_{1,1+q^2}$-orbit of any polynomial contains the $2$-cycle $\{1,q\}$ or one of the two fixed points $0$ or $1+q$. Furthermore, $T_{1,1+q^2}(x)$ is a polynomial if and only if $x$ is a polynomial.\footnote{A similar theorem holds more generally for any $\Tab$, where $B$ is a polynomial multiple of $A+q^2$.}
\end{theorem}

Since the only $2$-cycle for $T$ is $\{1,2\}$ and no positive integer has a $T$-orbit that contains the fixed points $0$ or $-1$, it follows that for any conjugacy between $T$ and $T_{1,1+q^2}$, the positive integers whose $T$-orbit contains $1$ correspond to polynomials.  Thus, to prove the $3x+1$ conjecture, it suffices to prove the following.

\begin{conjecture}\label{conj:2}
  For any conjugacy $H\colon\adics\to\Fq$ between $T$ and $T_{1,1+q^2}$ and any positive integer $n$, the power series $H(n)$ is a polynomial.
\end{conjecture}

Equivalently, the $3x+1$ conjecture is true if and only if $\xi(H(n))$ is an integer for every positive integer $n$.  We provide the first few such values in Table \ref{tbl:conjugates2} for any parity-preserving conjugacy, $H$.  Given that Conjecture \ref{conj:2} is equivalent to the original $3x+1$ conjecture, it is not surprising that no integer sequence starting with the first twenty values of $\xi(H(n))$ appears in the OEIS \cite{sloane}.

Notice that arithmetic in $\Fq$ is somewhat simpler than it is in $\adics$, because addition and multiplication of their binary representations are binary arithmetic without carrying.  For example, in $\Fq$ we have $3+6=5$, $3\cdot 6=3\cdot 4+3\cdot 2=10$. In this notation, since $\xi(1+q)=3$, $\xi(1)=1$, and $\xi(q)=2$, the formula for $T_q$ becomes
  $$
    T_q(x)  =
      \begin{cases}
         \frac{x}{2}, & \text{if $x$ is even;} \\
         \frac{3x+1}{2}, & \text{otherwise}
      \end{cases}
  $$
which looks exactly like the formula for $T$ except that the addition and multiplication are the corresponding binary operations without carrying. Similarly, the formula for $T_{1,1+q^2}$ becomes
  $$
    T_{1,1+q^2}(x)  =
      \begin{cases}
         \frac{x}{2}, & \text{if $x$ is even;} \\
         \frac{x+5}{2}, & \text{otherwise}
      \end{cases}
  $$
with the same simpler addition and multiplication.

As the results of Hicks, Mullen, Yucas, and Zavislak \cite{Hicks}, Snapp and Tracy \cite{Snapp}, and Theorem \ref{thm:main3} show, working with these simpler binary operations can allow us to more easily determine the behavior of the orbits.  Thus, determining a closed formula for the correspondence shown in either Table \ref{tbl:conjugates1} or Table \ref{tbl:conjugates2} could very well lead to a solution to the original conjecture itself.

\subsection{Proof of Theorem \ref{thm:main}}\label{sec:proofs}

It is well known (cf.\ Lagarias \cite{Lagarias}) that $T$ is conjugate to the \emph{shift map} $\sigma\colon\adics\to\adics$, given by
  $$\sigma(x)=\begin{cases}
                 \frac{x}{2}, &  \text{$x$ is even;} \\
                 \frac{x-1}{2}, & \text{$x$ is odd.}
               \end{cases}
        $$
Notice that if $x_0x_1x_2\cdots$ is the binary representation of $x$, then $\sigma(x_0x_1x_2\cdots)=x_1x_2x_3\cdots$, i.e., $\sigma$ ``shifts'' the binary digits by deleting the first one.  This makes it very easy to determine, for example, the number of cycles with $n$ elements for a fixed $n$.  In particular, $\sigma$ has two fixed points $\Set{\overline{0}}$ and $\Set{\overline{1}}$, a unique cycle of size two $\Set{\overline{01}, \overline{10}}$, and two cycles of size three $\Set{\overline{110}, \overline{011},\overline{101}}$ and $\Set{\overline{001}, \overline{100},\overline{010}}$.

Since conjugacy is an isomorphism in the category of discrete dynamical systems, to prove Theorem \ref{thm:main}, it suffices to show that $\Tab$ is conjugate to $\sigma$.  We begin with a few easy lemmas.

\begin{lemma}\label{lemma:factor}
  Let $n\in\bbN$ and $U,V\in \Fq$ with $V$ odd.  Then 
    $$\Tab(U+q^{n+1}V)=\Tab(U)+q^nW$$ 
  for some odd $W\in\Fq$.
\end{lemma}

\begin{proof}
  Let $n\in\bbN$ and $U,V\in \Fq$ with $V$ odd. Either $U$ is even or $U$ is odd.

  Case 1: Assume $U$ is even.  Then $U+Vq^{n+1}$ is even as well. So 
  \begin{align*}
    \Tab(U+q^{n+1}V) &= \frac{U+q^{n+1}V}{q} \\
                        &= \frac{U}{q}+q^nV \\
                        &= \Tab(U)+q^nV
  \end{align*}
  and $V$ is odd.
  
  Case 2: Assume $U$ is odd.  Then $U+Vq^{n+1}$ is odd as well. So 
  \begin{align*}
    \Tab(U+q^{n+1}V) &= \frac{A(U+q^{n+1}V)+B}{q} \\
                        &= \frac{AU+B}{q}+q^nAV \\
                        &= \Tab(U)+q^nAV,
  \end{align*}
  and $AV$ is odd since both $A$ and $V$ are.
  
  So in both cases, the result follows.
\end{proof}

\begin{corollary}\label{cor-tail}
  Let $n\in\bbN$ and $U,V\in \Fq$ with $V$ odd.  Then for all $k\leq n+1$ 
    $$\Tab^k(U+q^{n+1}V)=\Tab^k(U)+q^{n+1-k}W_k$$ 
  for some odd $W_k\in\Fq$.
\end{corollary}

\begin{proof}
  We proceed by induction on $k$ for arbitrary $n$. For the base case, we have
  \begin{align*}
      T^0_{A,B}(U+q^{n+1}V) &= U+q^{n+1}V \\
                            &= T^0_{A,B}(U)+q^{n+1-0}V
  \end{align*}
  where $V$ is odd.

  For the inductive step, let $k\leq n$ and assume 
  $$T^k_{A,B}(U+q^{n+1}V)=T^k_{A,B}(U)+q^{n+1-k}W_k$$
  for some odd $W_k$.  Then
  \begin{align*}
    T^{k+1}_{A,B}(U+q^{n+1}V) 
      &= \Tab(T^k_{A,B}(U+q^{n+1}V)) \\
      &= \Tab(T^k_{A,B}(U)+q^{n+1-k}W_k) \\
      &= T^{k+1}_{A,B}(U)+q^{n-k}W_{k+1} \\
      &= T^{k+1}_{A,B}(U)+q^{n+1-(k+1)}W_{k+1}
  \end{align*}
  for some odd $W_{k+1}$ by Lemma \ref{lemma:factor}, which completes the induction.
\end{proof}

This gives us the next interesting corollary, which essentially says that changing the coefficient of $q^n$ in a $q$-series changes the parity of its $n^\text{th}$ iterate. We write $\xcong{m}$ to indicate congruence with modulus $m$.

\begin{corollary}\label{cor-not-tail}
  Let $n\in\bbN$ and $U,V\in\Fq$ with $V$ odd.  Then $\Tab^n(U+q^nV)$ and $\Tab^n(U)$ have opposite parity, i.e., 
  $$\Tab^n(U+q^nV)\xncong{q}\Tab^n(U).$$
\end{corollary}

\begin{proof}
  We consider two cases, either $n=0$ or $n>0$.
  
  If $n=0$ then since $V$ is odd, $U+V \xncong{q} U$, and so
    \begin{align*}
       \Tab^n(U+q^nV) &= \Tab^0(U+q^0V) \\
                   &= U+V \\
                   &\xncong{q} U \\
                   &= \Tab^0(U). 
    \end{align*}
  
  If $n>0$, then by the previous corollary, there exists an odd $W\in\Fq$ such that $T^n_{A,B}(U+q^nV)=\Tab^n(U)+q^{n-n}W$, so
    \begin{align*}
      \Tab^n(U+q^nV) &= \Tab^n(U)+q^{n-n}W \\
                        &= \Tab^n(U)+W \\
                        &\xncong{q}\Tab^n(U).
    \end{align*}
\end{proof}

\begin{definition}
   For $x\in\Fq$, define \emph{the parity vector} $\PV_{A,B}(x)$ to be the $2$-adic integer whose binary representation is $x_0x_1x_2\cdots$ where for each $k\in\bbN$, 
   $x_k\in\{0,1\}$ is the parity of $T^k_{A,B}(x)$.
\end{definition}

We abbreviate $\PV_{A,B}$ as $\PV$ when the subscripts are clear from the context.

\begin{theorem}\label{thm-solenoidal}
    For $n\in\bbN$ and $a,b\in\Fq$,
    $$\PV(a)\xcong{2^{n+1}}\PV(b) \xiff a\xcong{q^{n+1}}b.$$
\end{theorem}

\begin{proof}
    We proceed by induction on $n$.
    
    For the base case, if $n=0$, then for any $a,b\in\Fq$, $T^0_{A,B}(a)=a$ and $T^0_{A,B}(b)=b$. So $\PV(a)\xcong{2^{0+1}}\PV(b)$ if and only if $a\xcong{q^{0+1}}b$.

    For the inductive hypothesis, let $k\in\bbN$ and assume for all $a,b\in\Fq$,
       $$\PV(a)\xcong{2^{k+1}}\PV(b)\xiff a\xcong{q^{k+1}}b.$$
    
    Let $a,b\in\Fq$.

    $(\xlimplies)$ Assume $a\xcong{q^{k+2}}b$.  Then $a=p+q^{k+1}c$ and $b=p+q^{k+1}d$ for some $p,c,d\in\Fq$ such that $p$ is a polynomial of degree less than $k+1$ and $c\xcong{q}d$ (i.e., they have the same parity). Clearly, 
    \begin{align*}
      a &=p+q^{k+1}c \\
        &\xcong{q^{k+1}}p \\
        &\xcong{q^{k+1}}p+q^{k+1}d \\
        &=b,   
    \end{align*} 
    so by the inductive hypothesis $\PV(a)\xcong{2^{k+1}}\PV(b)$.  
    
    Additionally, by Corollary \ref{cor-tail}, 
    \begin{align*}
      T^{k+1}_{A,B}(a) &= T^{k+1}_{A,B}(p+q^{k+1}c) \\
                       &= \Tab^{k+1}(p)+q^{0}c \\
                       &= \Tab^{k+1}(p)+c \\
                       &\xcong{q}\Tab^{k+1}(p)+d \\
                       &= \Tab^{k+1}(p)+q^{0}d \\
                       &= T^{k+1}_{A,B}(p+q^{k+1}d) \\
                       &= T^{k+1}_{A,B}(b).
    \end{align*}
    Combining these two results shows $\PV(a)\xcong{2^{k+2}}\PV(b)$ as desired.

    $(\ximplies)$ Now assume $\PV(a)\xcong{2^{k+2}}\PV(b)$. Then $\PV(a)\xcong{2^{k+1}}\PV(b)$, so by the induction hypothesis $a\xcong{q^{k+1}}b$. 
    
    Assume $a\xncong{q^{k+2}}b$. Since $a\xcong{q^{k+1}}b$, it must be that $a=p+q^{k+1}c$ and $b=p+q^{k+1}d$ for some $p,c,d\in\Fq$ such that $p$ is a polynomial of degree less than $k+1$ and $c\xncong{q}d$ (i.e., $c,d$ have opposite parity). Thus, again by Corollary \ref{cor-tail}, 
    \begin{align*}
      T^{k+1}_{A,B}(a) &= T^{k+1}_{A,B}(p+q^{k+1}c) \\
                       &= \Tab^{k+1}(p)+q^{0}c \\
                       &= \Tab^{k+1}(p)+c \\
                       &\xncong{q}\Tab^{k+1}(p)+d \\
                       &= \Tab^{k+1}(p)+q^{0}d \\
                       &= T^{k+1}_{A,B}(p+q^{k+1}d) \\
                       &= T^{k+1}_{A,B}(b)
    \end{align*}
    so $\PV(a)\xncong{2^{k+2}}\PV(b)$, which is a contradiction.  So $a\xcong{q^{k+2}}b$ as desired.
\end{proof}

It follows that $PV$ induces a bijection from $\Fq/q^n$ to $\adics/2^n$ for any positive $n$.

\begin{corollary}\label{cor-permutation}
    Let $n\in\bbN^+$.  For any $a\in\Fq$, define $\wPV([a])=[\PV(a)]$ where $[a]$ is the equivalence class of $a$ in $\Fq/q^n$ and $[\PV(a)]$ is the equivalence class of $\PV(a)$ in $\adics/2^n$.  Then $\wPV\colon \Fq/q^n\to\adics/2^n$ is a bijection.
\end{corollary}

\begin{proof}
    By Theorem \ref{thm-solenoidal}, $\wPV$ maps to the same equivalence class for all elements in the same equivalence class, so $\wPV$ is a well-defined function. 
    
    If $[a],[b]\in\Fq/q^n$ and $\wPV([a])=\wPV([b])$, then 
    \begin{align*}
        [\PV(a)] &= \wPV([a]) \\
                 &= \wPV([b]) \\
                 &= [\PV(b)].
    \end{align*}
    Thus, $\PV(a)\xcong{2^n}\PV(b)$. Therefore, $a\xcong{q^n}b$ by Theorem \ref{thm-solenoidal} and thus $[a]=[b]$.  So $\wPV$ is injective.

    Finally, both $\Fq/q^n$ and $\adics/2^n$ are finite sets with $2^n$ elements, so any injective function between them must be bijective as well.
\end{proof}

Taking the inverse limit of these maps with respect to $n$ produces the desired conjugacy needed to prove our main result.

\begin{proof}[Proof (of Theorem \ref{thm:main})]
    We will show that $\PV\colon\Fq\to\adics$ is a conjugacy between $\Tab$ and the shift map $\sigma$.

    Let $a,b\in\Fq$ and assume $\PV(a)=\PV(b)$.  Then for all positive $n$, $\PV(a)\xcong{2^n}\PV(b)$ and thus $a\xcong{q^n}b$ by Theorem \ref{thm-solenoidal}. But two elements of $\Fq$ which are congruent mod $q^n$ for all $n$ must be equal.  Therefore, $a=b$ and so $\PV$ is injective.

    To see that $\PV$ is surjective, let $x\in\adics$. Then by Corollary \ref{cor-permutation}, for each positive integer $n$ there exists $a_n\in\Fq$ such that $\PV(a_n)\xcong{2^n}x$.  Since $s\xcong{2^n}x$ implies $s\xcong{2^m}x$ for all $m\leq n$, by Theorem \ref{thm-solenoidal}, $a_n\xcong{q^n}a_m$ for all $m\leq n$. So there is a unique $a\in\Fq$ such that $a\xcong{q^n}a_n$ and $\PV(a)\xcong{2^n}x$ for all $n$. Thus, $\PV(a)=x$ as desired.

    Finally, to see that $\PV$ is a conjugacy, let $a\in\Fq$.  Then by the definition of parity vector,
    \begin{align*}
        (\PV\circ \Tab)(a) &= \PV(\Tab(a)) \\
                           &= \sigma(\PV(a))  \\
                           &= (\sigma\circ\PV)(a).
    \end{align*}
    
    So $\Tab$ is conjugate to the shift map via conjugacy $\PV$.  But it is well known (cf.\ Lagarias \cite{Lagarias}) that the shift map is conjugate to $T$ with conjugacy $\Phi\colon \adics\to\adics$, and so $\Tab$ is conjugate to $T$ with conjugacy $h=\Phi\circ\PV$.
\end{proof}

This proves Theorem \ref{thm:main}. Corollary \ref{cor:main1} follows immediately as a special case since $T_q=T_{1+q,1}$.

\subsection{Proof of Theorem \ref{thm:main3}}

To determine which $\Tab$ are conjugate to $T$ via a conjugacy that maps polynomials to positive integers, we can first ask which $\Tab$ have a polynomial $2$-cycle, since such a $2$-cycle will be unique and correspond to the unique $2$-cycle $\{1,2\}$ for $T$.

\begin{lemma}\label{lem:poly-cycles}
  The unique $2$-cycle for $\Tab$ is $\left\{\frac{B}{A+q^2},\frac{qB}{A+q^2}\right\}$, and its fixed points are $0$ and $\frac{B}{A+q}$. 
\end{lemma}

\begin{proof}
  Since $\Tab$ is conjugate to the shift map, it must also have a unique $2$-cycle and two fixed points.  These elements must all satisfy $\Tab(\Tab(x))=x$.  If both $x$ and $\Tab(x)$ are even, then $x=\Tab(\Tab(x))=(x/q)/q$, so $x=0$. If both $x$ and $\Tab(x)$ are odd, then $x=\Tab(\Tab(x))=(A((Ax+B)/q)+B)/q$, so $x=\frac{B}{A+q}$.  Finally, if $x$ is even and $\Tab(x)$ is odd, then $x=\Tab(\Tab(x))=(A((x/q)+B)/q$, so $x=\frac{qB}{A+q^2}$ and $\Tab(x)=\frac{B}{A+q^2}$ (reversing the parities just swaps $x$ and $\Tab(x)$).
\end{proof}

Thus, the $2$-cycle of $\Tab$ will consist of polynomials whenever $\frac{B}{A+q^2}$ is a polynomial, i.e., whenever $B$ is a polynomial multiple of $A+q^2$.  The simplest example of this family of dynamical systems is when $A=1$ and $B=1+q^2$. In that case, the $2$-cycle of $\Tqq$ is $\{1,q\}$ and its fixed points are $0$ and $1+q$.  Thus, to prove Theorem \ref{thm:main3} it suffices to show that every polynomial eventually enters either the $2$-cycle or one of the two fixed points, and that $\Tqq$ maps every polynomial, and only polynomials, to other polynomials.

\begin{proof}[Proof (of Theorem \ref{thm:main3})]
Let $x\in\Fq$ be a polynomial of degree $n>1$. If $x$ is even then $\Tqq(x)=x/q$ is a polynomial of degree $n-1$.  If $x$ is odd then $x=P+q^n$ for some odd polynomial $P$ of degree less than $n$ (i.e., $P=x+q^n$). So 
\begin{align*}
    \Tqq(x) &=\Tqq(P+q^n) \\
            &=\frac{(P+q^n)+1+q^2}{q} \\
            &=\frac{1+P}{q}+q+q^{n-1}.
\end{align*}

But $(1+P)/q+q+q^{n-1}$ is either $0$ or is a polynomial of degree less than $n$ because $P$ is odd.  

In both cases, the degree of $\Tqq(x)$ is strictly less than the degree of $x$.  So the $\Tqq$-orbit of every polynomial of degree greater than $1$ is a sequence of polynomials of decreasing degrees and thus must eventually reach $0$, $1$, $q$, or $1+q$, which are the two fixed points and the two elements of the $2$-cycle as claimed.

For the second part of the theorem, let $x\in\Fq$. If $x$ is an even polynomial, then $x=qP$ for some polynomial $P$, so $\Tab(x)=x/q=qP/q=P$ is a polynomial. If $x$ is odd then $x=1+qP$ for some polynomial $P$, so $\Tqq(x)=\Tqq(1+qP)=(1+qP+1+q^2)/q=P+q$ which is also a polynomial.  

Conversely, let $x\in\Fq$ such that $\Tqq(x)$ is a polynomial. If $x$ is even then $\Tqq(x)=x/q$, so $x=q\Tqq(x)$ is a polynomial.  Similarly, if $x$ is odd then $\Tqq(x)=(x+1+q^2)/q$ so $x=q\Tqq(x)+1+q^2$ which, once again, is a polynomial.

This completes the proof.
\end{proof}

In conclusion, recent authors have shown that the map $T_q$ on $F_2[q]$ has the property that the $T_q$-orbit of every polynomial contains $1$. This does not prove the original $3x+1$ conjecture, however, because the correspondence $\xi$ between nonzero polynomials and positive integers is not a conjugacy of discrete dynamical systems.  

By extending the domains of $T_q$ and $T$ to $\Fq$ and $\adics$, respectively, we find that $T_q$ and $T$ are actually conjugate as discrete dynamical systems, but not via a conjugacy that maps polynomials to positive integers. This result extends to an infinite family of maps on $\Fq$ which are all conjugate to $T$, and some of these, such as $\Tqq$, do have the property that the positive integers whose $T$-orbit contains $1$ correspond to polynomials.

By applying the correspondence $\xi$ to the polynomials that correspond to the positive integers whose $T$-orbit contains $1$, we obtain the sequence $\xi(H(n))$ that starts with 
$$1,2,47,4,21,94,2763,8,11049,42,1383,188,173,5526,5427,16,689,22098,22143,84,\ldots$$
Showing that every term in this sequence is an integer, or equivalently, proving that every positive integer corresponds to a polynomial via a conjugacy with $\Tqq$, would resolve the original $3x+1$ conjecture.

\begin{table}
  \centering
  \begin{tabular}{lrl}
    \toprule
     $n$ & $\xi(h(n))$ & $h(n)$ \\
    \midrule 
    \addlinespace[0.5em]
    $ 1$ & $-3/7$ & $(1+q)\cdot\frac{1}{(1+q^{3})}$ \\ 
    \addlinespace[0.5em]
    $ 2$ & $-6/7$ & $(q+q^{2})\cdot\frac{1}{(1+q^{3})}$ \\
    \addlinespace[0.5em]
    $ 3$ & $-13/7$ & $1+(1+q^{2})\cdot\frac{q^{2}}{(1+q^{3})}$ \\
    \addlinespace[0.5em]
    $ 4$ & $-12/7$ & $(q+q^{2})\cdot\frac{q}{(1+q^{3})}$ \\
    \addlinespace[0.5em]
    $ 5$ & $9/7$ & $1+q+(1+q)\cdot\frac{q^{2}}{(1+q^{3})}$ \\
    \addlinespace[0.5em]
    $ 6$ & $-26/7$ & $q+(1+q^{2})\cdot\frac{q^{3}}{(1+q^{3})}$ \\
    \addlinespace[0.5em]
    $ 7$ & $-175/9$ & $1+q^{3}+(q^{3}+q^{4}+q^{5})\cdot\frac{q^{5}}{(1+q^{6})}$ \\
    \addlinespace[0.5em]
    $ 8$ & $-24/7$ & $(q+q^{2})\cdot\frac{q^{2}}{(1+q^{3})}$ \\
    \addlinespace[0.5em]
    $ 9$ & $1187/65$ & $1+q+q^{5}+(1+q+q^{2}+q^{3}+q^{5}+q^{10})\cdot\frac{q^{6}}{(1+q^{12})}$ \\
    \addlinespace[0.5em]
    $10$ & $18/7$ & $q+q^{2}+(1+q)\cdot\frac{q^{3}}{(1+q^{3})}$ \\
    \addlinespace[0.5em]
    $11$ & $-37/7$ & $1+q^{2}+q^{3}+q^{2}\cdot\frac{q^{5}}{(1+q^{3})}$ \\
    \addlinespace[0.5em]
    $12$ & $-52/7$ & $q^{2}+(1+q^{2})\cdot\frac{q^{4}}{(1+q^{3})}$ \\
    \addlinespace[0.5em]
    $13$ & $-31/7$ & $1+q+q^{2}+(1+q^{2})\cdot\frac{q^{4}}{(1+q^{3})}$ \\
    \addlinespace[0.5em]
    $14$ & $-350/9$ & $q+q^{4}+(q^{3}+q^{4}+q^{5})\cdot\frac{q^{6}}{(1+q^{6})}$ \\
    \addlinespace[0.5em]
    $15$ & $-359/9$ & $1+q^{4}+(q^{3}+q^{4}+q^{5})\cdot\frac{q^{6}}{(1+q^{6})}$ \\
    \addlinespace[0.5em]
    $16$ & $-48/7$ & $(q+q^{2})\cdot\frac{q^{3}}{(1+q^{3})}$ \\
    \addlinespace[0.5em]
    $17$ & $-115/7$ & $1+q+q^{3}+(q+q^{2})\cdot\frac{q^{5}}{(1+q^{3})}$ \\
    \addlinespace[0.5em]
    $18$ & $2374/65$ & $q+q^{2}+q^{6}+(1+q+q^{2}+q^{3}+q^{5}+q^{10})\cdot\frac{q^{7}}{(1+q^{12})}$ \\
    \addlinespace[0.5em]
    $19$ & $2309/65$ & $1+q^{2}+q^{6}+(1+q+q^{2}+q^{3}+q^{5}+q^{10})\cdot\frac{q^{7}}{(1+q^{12})}$ \\
    \addlinespace[0.5em]
    $20$ & $36/7$ & $q^{2}+q^{3}+(1+q)\cdot\frac{q^{4}}{(1+q^{3})}$ \\
    \addlinespace[0.5em]
   \bottomrule
  \end{tabular}
  \caption{The power series in $\Fq$ that correspond to the first twenty positive integers $n$ via any parity-preserving conjugacy $h\colon\adics\to\Fq$ between $T$ and $T_q$.}
  \label{tbl:conjugates1}
\end{table}

\begin{table}
  \centering
  \begin{tabular}{lrl}
    \toprule
     $n$ & $\xi(H(n))$ & $H(n)$ \\ 
    \midrule 
    $ 1$ & $    1$ & $1$ \\ \addlinespace[0.5em]
    $ 2$ & $    2$ & $q$ \\ \addlinespace[0.5em]
    $ 3$ & $   47$ & $q^5+q^3+q^2+q+1$ \\ \addlinespace[0.5em]
    $ 4$ & $    4$ & $q^2$ \\ \addlinespace[0.5em]
    $ 5$ & $   21$ & $q^4+q^2+1$ \\ \addlinespace[0.5em]
    $ 6$ & $   94$ & $q^6+q^4+q^3+q^2+q$ \\ \addlinespace[0.5em]
    $ 7$ & $ 2763$ & $q^{11}+q^9+q^7+q^6+q^3+q+1$ \\ \addlinespace[0.5em]
    $ 8$ & $    8$ & $q^3$ \\ \addlinespace[0.5em]
    $ 9$ & $11049$ & $q^{13}+q^{11}+q^9+q^8+q^5+q^3+1$ \\ \addlinespace[0.5em]
    $10$ & $   42$ & $q^5+q^3+q$ \\ \addlinespace[0.5em]
    $11$ & $ 1383$ & $q^{10}+q^8+q^6+q^5+q^2+q+1$ \\ \addlinespace[0.5em]
    $12$ & $  188$ & $q^7+q^5+q^4+q^3+q^2$ \\ \addlinespace[0.5em]
    $13$ & $  173$ & $q^7+q^5+q^3+q^2+1$ \\ \addlinespace[0.5em]
    $14$ & $ 5526$ & $q^{12}+q^{10}+q^8+q^7+q^4+q^2+q$ \\ \addlinespace[0.5em]
    $15$ & $ 5427$ & $q^{12}+q^{10}+q^8+q^5+q^4+q+1$ \\ \addlinespace[0.5em]
    $16$ & $   16$ & $q^4$ \\ \addlinespace[0.5em]
    $17$ & $  689$ & $q^9+q^7+q^5+q^4+1$ \\ \addlinespace[0.5em]
    $18$ & $22098$ & $q^{14}+q^{12}+q^{10}+q^9+q^6+q^4+q$ \\ \addlinespace[0.5em]
    $19$ & $22143$ & $q^{14}+q^{12}+q^{10}+q^9+q^6+q^5+q^4+q^3+q^2+q+1$ \\ \addlinespace[0.5em]
    $20$ & $   84$ & $q^6+q^4+q^2$ \\ \addlinespace[0.5em]
    \bottomrule \addlinespace[0.5em]
  \end{tabular}
  \caption{The polynomials in $\Fq$ that correspond to the first twenty positive integers $n$ via any parity-preserving conjugacy $H\colon\adics\to\Fq$ between $T$ and $T_{1,1+q^2}$.}
  \label{tbl:conjugates2}
\end{table}

\bigskip
\hrule
\bigskip

\noindent 2020 {\it Mathematics Subject Classification}:
Primary 11B83; Secondary 11B37, 37B10.

\noindent \emph{Keywords: } Collatz conjecture, $q$-series.

\bigskip
\hrule
\bigskip

\end{document}